\documentclass[12pt]{article}
 \usepackage{a4, latexsym, amsfonts,amsmath}
 \newtheorem{thm}{Theorem}[section]
\newtheorem{lem}{Lemma}[section]
 \newtheorem{defn}{Definition}[section]
 \newtheorem{coro}{Corollary}[section]

 \newtheorem{open problem}{Open problem}[section]
 \newtheorem{prop}{Proposition}[section]
 
 \newtheorem{conj}{Conjecture}[section]
 \newcommand{\qed}{\nopagebreak\hspace*{\fill}$\square$\par\vskip2mm}

 \newenvironment{proof}{\trivlist
      \item[\hskip\labelsep
      {\itshape Proof.}]\normalfont}
      {\hspace*{\fill}$\Box$\endtrivlist}

\begin{document}

\title{ Purely singular splittings of  cyclic groups}


\author{Kevin Zhao \thanks{K. Zhao is with School of  of Mathematical Science, South China Normal University,  Guangzhou 510631, China (email:zhkw-hebei@163.com)},
Pingzhi Yuan\thanks{Corresponding author. P. Yuan is with School of  of Mathematical Science, South China Normal University,  Guangzhou 510631, China (email: yuanpz@scnu.edu.cn).} }

\date{}
\maketitle
 \edef \tmp {\the \catcode`@}
   \catcode`@=11
   \def \@thefnmark {}

    \@footnotetext { Supported by  NSF of China  (Grant No. 11671153) and NSF of Guangdong (No. 2016A030313850).}

    \begin{abstract} Let $G$ be a finite abelian group.
  We  say that $M$ and $S$ form a \textsl{splitting} of $G$ if every nonzero element $g$ of $G$ has a unique representation of the form $g=ms$ with $m\in M$ and $s\in S$, while $0$ has no such representation.
The splitting is called \textit{purely singular}
if for each prime divisor $p$ of $|G|$, there is at least one element of $M$ is divisible by $p$.

  In this paper, we mainly study the purely singular splittings of cyclic groups.  We first prove that if $k\ge3$ is a positive integer such that $[-k+1, \,k]^*$ splits a cyclic group $\mathbb{Z}_m$, then $m=2k$. Next, we have the following general result. Suppose  $M=[-k_1, \,k_2]^*$ splits $\mathbb{Z}_{n(k_1+k_2)+1}$ with $1\leq k_1< k_2$.
If $n\geq 2$, then $k_1\leq n-2$ and $k_2\leq 2n-5$. Applying this result, we prove that if $M=[-k_1, \,k_2]^*$ splits $\mathbb{Z}_m$ purely singularly, and  either $(i)$ $\gcd(s, \,m)=1$ for all $s\in S$ or $(ii)$ $m=2^{\alpha}p^{\beta}$ or $2^{\alpha}p_1p_2$
with $\alpha\geq 0$, $\beta\geq 1$ and $p$, $p_1$, $p_2$ odd primes,
then $m=k_1+k_2+1$ or $k_1=0$ and $m=k_2+1$ or $2k_2+1$.

\end{abstract}


\section{Introduction}

Let $G$ be a finite group, written additively, $M$ a set of integers, and $S$ a subset of $G$.
We  say that $M$ and $S$ form a \textsl{splitting} of $G$ if every nonzero element $g$ of $G$ has a unique representation of the form $g=ms$ with $m\in M$ and $s\in S$, while $0$ has no such representation.
(Here''$ms$" denotes the sum of $m$ $s$'s if $m\geq 0$, and $-((-m)s)$ if $m<0$.)
We  write ''$G\setminus \{0\}=MS$" to indicate that $M$ and $S$ form a splitting of $G$.
$M$  is referred to as the multiplier set and $S$ as the splitter set.
We  also say that $M$ splits $G$ with  a splitter set $S$, or simply that $M$ splits $G$.

 Let $a, b$ be integers such that $a\le b$, denote
$$[a, b] = \{a, a + 1, a + 2, \ldots , b\}\,\, \mbox{and}\,\,\,
[a, b]^\ast = \{a, a + 1, a + 2, \ldots, b\}\backslash\{0\}.$$
For any positive integer $q$, let $\mathbb{Z}_q$ be the ring of integers
modulo $q$ and $\mathbb{Z}_q ^\ast= \mathbb{Z}_q \backslash \{0\}$. For $a\in\mathbb{Z}_q ^\ast$, $o(a)$ denotes the order of $a$ in the multiplicative group $\mathbb{Z}_q ^\ast$.

Let $q$ be a positive integer and $k_1, k_2$ be non-negative
integers with $0\le  k_1\le  k_2$. The set $B\subset\mathbb{Z}_q$ of size $n$ is
called a {\it splitter} {\it set} (or a {\it packing} {\it set}) if all the sets
$$\{ab\pmod{q}: a\in [-k_1, k_2]\},\quad  b\in B $$
have $k_1+k_2$ nonzero elements, and they are disjoint. We denote
such a splitter set by $B[-k_1, k_2](q)$ set.
A $B[-k_1, k_2](q)$ set is called {\it perfect} if $n = \frac{q- 1}{k_1 + k_2}$. Clearly,
a perfect set can exist only if $q\equiv1\pmod{ k_1 + k_2}$.   A perfect $B[-k_1, k_2](q)$ set
is called {\it nonsingular} if $\gcd(q, k_2!) = 1$. Otherwise, the set is called
{\it singular}. If for any prime $p|q$, there is some $k$ with $0 < k \le k_2$
such that $p|k$, then the perfect $B[-k_1, k_2](q)$ set is called {\it
purely} {\it singular}.

{\bf Remark:}  Let $q$ be a positive integer and $k_1, k_2$ be non-negative
integers with $0\le  k_1\le  k_2$. Let $M=[-k_1, k_2]^\ast$. Then $B$ is a perfect $B[-k_1, k_2](q)$ set if and only if $MB$ is a splitting of $\mathbb{Z}_q$ by the works of D. Hickerson \cite{[H]} and Schwarz \cite{[SC]}. Therefore, we are only interested in considering  purely singular perfect $B[-k_1, k_2](q)$ sets for the cyclic group $\mathbb{Z}_q$ and nonsingular  perfect $B[-k_1, k_2](p)$ sets for an odd prime $p$.

In this paper, we focus our attention to the purely singular perfect $B[-k_1, \,k_2](q)$ sets for the cyclic group $\mathbb{Z}_q$. Zhang and Ge \cite{[ZG]} proposed the following conjecture.

\begin{conj} \label{conj}
Let $k_1, \,k_2$ be integers with $1 \le k_1 < k_2$ and
$k_1 + k_2\ge 4$, then there does not exist any purely singular
perfect $B[-k_1, \,k_2](m)$ set except for $m = 1$ and except possibly for $m = k_1 + k_2 + 1$.
\end{conj}

Zhang and Ge \cite{[ZG]} proved that Conjecture  \ref{conj} holds for $[-1, \,k]^\ast$ when $k=3, \,4, \,5$, $6, \,8$, $9, \,10$ and $[-2, \,k]^\ast$ when $k=3, \,4, \,6$. The authors \cite{[PK]} obtain some results on the purely singular
perfect $B[-k_1, \,k_2](q)$ sets and showed that Conjecture  \ref{conj} holds
 for $q=2^n$.

For the case when $k_1=0$, we have the following conjecture of Woldar \cite{[W]}.

\begin{conj} \label{conj1} Let $k$ be a positive integer. If $[1, \, k]$ splits the finite abelian group $G$ purely singularly, then $G$
is one of $\mathbb{Z}_1$, $\mathbb{Z}_{k + 1}$, or $\mathbb{Z}_{2k + 1}$.\end{conj}

Conjecture \ref{conj1} has been verified by Hickerson \cite{[W]} for all $k < 3000$. In this paper, by using some technique of the paper \cite{[PK]}, we first prove that  Conjecture  \ref{conj} holds for $[-k+1, \,k]^\ast$ when $k\ge3$. We have

\begin{thm} \label{minus=1-splitting}
Let $G$ be a finite cyclic group and $k\ge3$ a positive integer. Then $[-k+1, \,k]^*$ splits $G$ if and only if $G$ is a cyclic group of order $2k$.
\end{thm}
Next we prove a theorem which is very useful in the proof of Theorems \ref{primeB} and \ref{finally-thm}.

\begin{thm} \label{thm1}
Let $n$, $k_1$ and $k_2$ be positive integers with $n\geq 2$, $1\leq k_1< k_2$.
If $[-k_1, \,k_2]^*$ splits $\mathbb{Z}_{n(k_1+k_2)+1}$,
then $k_1\leq n-2$ and $k_2\leq 2n-5$.
\end{thm}

Finally we prove that   Conjectures \ref{conj} and \ref{conj1} hold for  cyclic groups $\mathbb{Z}_q$ of various $q$, we have

\begin{thm} \label{primeB}
Let $k_1$, $k_2$, $m$ be integers with $0\leq k_1\leq k_2$ and $k_2\ge3$, and let $\mathbb{Z}_m\setminus \{0\}=[-k_1,k_2]^*\cdot S$ be a splitting of the cyclic group $\mathbb{Z}_m$ with the splitter set $S$. If $\gcd(s, \,m)=1$ for all $s\in S$, then either $m=k_1+k_2+1$ or $k_1=0$ and $m=2k_2+1$.
\end{thm}

\begin{thm} \label{finally-thm}
Let $\alpha$, $\beta$, $k_1$, $k_2$ be integers, $1\leq k_1\leq k_2$ and  $k_2\ge3$. 
Suppose $[-k_1, \,k_2]^*$ splits a cyclic group $\mathbb{Z}_{m}$.
If the splitting is purely singular and $m=2^{\alpha}p^{\beta}$ or $2^{\alpha}p_1p_2$
with $\alpha\geq 0$, $\beta\geq 1$ and $p$, $p_1$, $p_2$ are odd primes,
then $m=k_1+k_2+1$. 

Furthermore, if $k_1=0$, then either $m=k_2+1$ or $m=2k_2+1$.
\end{thm}

\begin{thm}Let $k_1$, $k_2$, $m$ be integers with $0\leq k_1\leq k_2$ and $4\le k_1+k_2\le14$. Then there does not exist any  purely singular perfect $B[-k_1, \,k_2](m)$ set except for $m = 1$  and except possibly for $m = k_1 + k_2 + 1$,  and except possibly for $m=2(k_1+k_2)+1$ when $k_1=0$. In particular,  Conjecture \ref{conj} holds for $k_1+k_2\le14$. \end{thm}

\section{Preliminaries}

An equivalence between lattice tilings and Abelian-group splittings was described in [1, \, 2, \,3]. In \cite{[SAT]}, there are two important lemmas (in the language of splitting).

\begin{lem}[\cite{[SAT]}, Theorem 2] \label{cross}
If $n\geq 2$ and the $[-k, \, k]^\ast$ splits the cyclic group $\mathbb{Z}_{2kn+1}$, then $k\leq n-1$.
\end{lem}

\begin{lem}[\cite{[SAT]}, Theorem 3] \label{semi-cross}
If $n\geq 3$ and the $[1, \, k]$ splits the cyclic group $\mathbb{Z}_{kn+1}$, then $k\leq n-2$.
\end{lem}

To prove our main theorems, we need some lemmas.
The following lemma will be used repeatedly.

\begin{lem}[\cite{[ZG]}, Lemma 2] \label{spec-cyclic}
If $m|n$ and there exist both a perfect $B[-k_1, \,k_2](m)$ set and a perfect $B[-k_1, \,k_2](n)$ set,
then there exists a perfect $B[-k_1, \,k_2](n/m)$ set.
\end{lem}

We also need the following lemma (\cite{[W]} Theorem 6).

\begin{lem} \label{gen-plus=n-1} Suppose $[1, \,k]$ splits $\mathbb{Z}_m$ with $2k + 1$ composite. Then either
\begin{description}
  \item[(i)] $\gcd(2k + 1, \,m) = 1$, or
  \item[(ii)] $2k + 1$ divides $m$ and $\gcd(2k + 1, \frac{m}{2k+1}) = 1$.
\end{description}

\end{lem}

Combining Lemma 2.5 \cite{[PK]} and  Theorem $5$ in \cite{[W]}  yields the following lemma.
\begin{lem} \label{gen-pk-lem}
Let $k_1,$ $k_2$ be integers, $0\leq k_1\leq k_2$. Suppose there exists a perfect $B[-k_1, \,k_2](m)$ set with $k_1 + k_2 + 1$ composite.
If $k_1+k_2\geq 4$, $1\leq k_1\leq k_2$ or $k_1=0$, then either
\begin{description}
  \item[(i)] $\gcd(k_1 + k_2 + 1,m) = 1,$ or
  \item[(ii)] $k_1 + k_2 + 1|m$ and $\gcd(k_1 + k_2 + 1, \frac{m}{k_1+k_2+1}) = 1$.
\end{description}
\end{lem}

\begin{defn}Let $(G, \cdot)$ be an abelian group (written multiplicatively). If each element $g\in G$ can be expressed uniquely in the form
$$g = a \cdot b, a \in A,\,\, b \in B,$$
then the equation $G = A \cdot B$ is called a {\it factorization} of $G$. A non-empty subset of $G$ is called to be a {\it direct  factor} of $G$ if there exists a subset $B$ such that $G=A\cdot B$ is a factorization. \end{defn}

We also need the following result for the factorization of abelian groups.

\begin{prop}{\rm (\cite{SzS09}Theorem 7.12)}If $ G = A \cdot B$ is a factorization of the finite abelian group $G$ (written multiplicatively) and $k$ is an integer relatively prime to $|A|$, then $ G = A^k\cdot B$ is a factorization of the abelian group $G$, where $A^k=\{a^k : a \in A\}$.\end{prop}

\section{Proof of Theorem \ref{minus=1-splitting}}

In \cite{[H]}, the authors proved that if $G\setminus \{0\}=MS$ is a purely singular splitting with $|M|=3$,
then $G=\mathbb{Z}_{2^{2r}}$ for some $r\geq 0$; moreover Schwartz \cite{[SC]} has constructed an infinite family of
purely singular perfect $B[-1, 2](4^l)$ sets.


If $M=[-k+1, \,k]^*$ with $k\geq 3$, Schwartz \cite{[SC]} proved that $(k, \,|G|)\neq 1$.
We will prove Theorem 1.1.

{\bf Proof of Theorem 1.1:}
Suppose  $[-k+1,k]^*$ splits $\mathbb{Z}_q$ with the splitter set $S$ and $n=|S|$.
It suffices to show that $q=2k$.

We claim that if $q>1$, then $(k,q)>1$.
Let $S=\{a_1,\cdots ,a_n\}$.
Obviously, for any $1\leq u\leq n$, $-ka_u\in \mathbb{Z}_q\setminus \{0\}.$
Set $-ka_u=ia_j$ where $i\in [-k+1, \,k]^\ast$ and $1\leq j\leq n.$
If $i<k$, then we have $ka_u=-ia_j$, a contradiction.
Hence, $i=k$ and $k(a_u+a_j)=0$.
If $gcd(k, \,q)=1$, then $a_u=-a_j$, which is impossible.
Therefore, $(k, \,q)>1$, and so $(2k, \,q)>1$.

Since $(k-1)+k\geq 4$ ($k\geq 3$) and $(k-1)+k+1=2k$ is composite, by Lemma \ref{gen-pk-lem},
we have $2k|q$ and $(2k, \frac{q}{2k}) = 1$.
By Lemma \ref{spec-cyclic}, it is easy to see that $[-k+1, \,k]^*$ splits $\mathbb{Z}_{\frac{q}{2k}}$.
If $\frac{q}{2k}>1$, then by the claim, $(k, \,\frac{q}{2k})>1$, which is impossible.
The converse is obvious. This completes the proof.
\qed

\section{Proof of Theorem \ref{thm1}}

In this section, we will prove Theorem \ref{thm1}. Theorem \ref{thm1} can be viewed as a generalization of
Lemma \ref{cross} and \ref{semi-cross}.
To prove Theorem \ref{thm1}, we need some lemmas.

\begin{lem} \label{Elemmid}
Let $n$, $k$ and $l$ be integers with $n\geq 2$, $k\geq n-1$ and $l\geq 1$.
Suppose  $[-k, \,k+2l]^\ast$ splits the cyclic group $\mathbb{Z}_{n(2k+2l)+1}$.
Let $s$ and $s'$ be two elements of a splitter set.
Then one of these two conditions holds:
\begin{description}
  \item[(a)] There are integers $x$ and $y$, $1\leq x\leq 2n+2l-3$, $1\leq |y|\leq k$, such that $xs+ys'=0$;
  \item[(b)] $s'=\pm(2n+2l-2)s$, $k=n-1$ or $l=1$ and $s$ is a generator of $\mathbb{Z}_{n(2k+2l)+1}$.
\end{description}
\end{lem}

\begin{proof}
Define  a map $f:$ $\mathbb{Z}\oplus \mathbb{Z}\rightarrow \mathbb{Z}_{n(2k+2l)+1}$ by $f(i,j)=is+js'$.
Put $A=\{(i,j):0\leq i\leq 2n+2l-3, \, 0\leq j\leq k\}$.

If $f|_A:$ $A\rightarrow \mathbb{Z}_{n(2k+2l)+1}$ is not an injective map,
then there are two distinct elements $(i_1, \,j_1)$, $(i_2, \,j_2)$ of $A$ such that
$i_1s+j_1s'=i_2s+j_2s'$,  $i_1$,  $i_2\in [0,2n+2l-3]$ and $j_1$, $j_2\in [0,k]$. Note that
$i_1=i_2$ implies that $j_1s'=j_2s'$, which contradicts the fact that $[-k,k+2l]^*$ splits  $\mathbb{Z}_{n(2k+2l)+1}$.
If $j_1=j_2$, then $(i_1-i_2)s=0$ and hence $ord(s)\leq |i_1-i_2|\leq 2n+2l-3$.
Since $k\geq n-1$ and all elements $is$ for $i\in [-k, \,k+2l]^\ast$ are distinct and nonzero,
we have $ord(s)\geq 2k+2l+1\geq 2n+2l-1>2n+2l-3\geq ord(s)$, a contradiction.
Consequently, we may assume that $i_1<i_2$ and $j_1\neq j_2$.
Let $x=i_2-i_1$ and $y=j_2-j_1$.
Then $1\leq x\leq 2n+2l-3$, $1\leq |y|\leq k$ and $xs+ys'=0$. Hence condition (a) holds.

If $f|_A:$ $A\rightarrow \mathbb{Z}_{n(2k+2l)+1}$ is injective,
then $|A|=|f(A)|\leq |\mathbb{Z}_{n(2k+2l)+1}|=2n(k+l)+1$.
Since $|A|=(2n+2l-2)(k+1)=2n(k+l)+2(l-1)(k+1-n)\le 2n(k+l)+1$ and $k\ge n-1$,
we must have $k=n-1$ or $l=1$.
Thus $|A|=|f(A)|=2n(k+l)=|\mathbb{Z}_{n(2k+2l)+1}|-1$, say,  $$f(A)=\mathbb{Z}_{n(2k+2l)+1}\setminus\{x\}.$$
Obviously,  $\{x-s, \,x+s, \, x-s', \,x+s'\}\subseteq \mathbb{Z}_{n(2k+2l)+1}\setminus\{x\}$. Let
$$x-s=f(i, j)=is+js', \quad 0\le i\le 2n+2l-3, \,\, 0\le j\le k,$$
then $x=(i+1)s+js'\not\in f(A)$, which implies that $i=2n+2l-3$. Hence
\begin{gather}
x-s=f(2n+2l-3,j_1)=(2n+2l-3)s+j_1s', \quad j_1 \in [0,k]. \label{eq1}
\end{gather}
Similarly, there are $i_1$,  $i_2\in [0, \,2n+2l-3]$ and $j_2\in [0, \,k]$ such that
\begin{gather}
x+s=f(0, \,j_2)=j_2s';   \label{eq2} \\
x-s'=f(i_1, \,k)=i_1s+ks';   \label{eq3} \\
x+s'=f(i_2, \,0)=i_2s.    \label{eq4}
\end{gather}

Delete $x$ from (\ref{eq2}) and (\ref{eq4}), we obtain
\begin{gather}(i_2+1)s-(j_2+1)s'=0.   \label{eq5}
\end{gather}

Similarly, from (\ref{eq1}), (\ref{eq2}), (\ref{eq3}) and (\ref{eq4}) we conclude that
\begin{gather}
(2n+2l-1)s+(j_1-j_2)s'=0; \label{eq6} \\
(2n+2l-2-i_1)s+(j_1-(k+1))s'=0;   \label{eq7} \\
(2n+2l-2-i_2)s+(j_1+1)s'=0;   \label{eq8} \\
(i_1+1)s+(k+1-j_2)s'=0.    \label{eq9}
\end{gather}
If the condition (a) does not hold, since $i_2+1\le 2n+2l-2$ and $j_2+1\le k+1$, so $i_2=2n+2l-3$ or $j_2=k$ by (\ref{eq5}). Similarly, if the condition
 $(a)$ does not hold, then it follows from (\ref{eq2}), (\ref{eq3}) and (\ref{eq4}) that
$$i_1=0 \ or \ j_1=0;$$
$$i_2=0 \ or \ j_1=k;$$
$$i_1=2n+2l-3 \ or \ j_2=0.$$
If $i_1=2n+2l-3$, then $j_1=0$, $i_2=0$ and $j_2=k$.
By (\ref{eq9}), $$s'=-(2n+2l-2)s.$$
If $j_2=0$, then $i_2=2n+2l-3$, $j_1=k$ and $i_1=0$.
By (\ref{eq7}), $$s'=(2n+2l-2)s.$$

Since $s$ and $s'$ generate $\mathbb{Z}_{n(2k+2l)+1}$ and $s'=\pm(2n+2l-2)s,$ it follows that $s$ generates $\mathbb{Z}_{n(2k+2l)+1}$. This proves the lemma.
\end{proof}

\begin{lem} \label{Olemmid}
Let $n$, $k$ and $l$ be integers, $n\geq 2$, $k\geq n-1$ and $l\geq 0$.
Suppose that $[-k, \,k+2l+1]^*$ splits the cyclic group $\mathbb{Z}_{n(2k+2l+1)+1}$.
Let $s$ and $s'$ be elements of a splitter set.
Then one of the following statements holds:
\begin{description}
  \item[(a)] there are integers $x$ and $y$, $1\leq x\leq 2n+2l-2$, $1\leq |y|\leq k$, such that $xs+ys'=0$;
  \item[(b)] $s'=\pm(2n+2l-1)s$ and $G$ is cyclic with generator $s$ and $k=n-1$.
\end{description}
\end{lem}

\begin{proof} We prove the lemma by the similar argument as in the proof of Lemma \ref{Elemmid}.
Define a map $f:$ $Z\oplus Z\rightarrow G$ by $f(i, \,j)=is+js'$.
Let $A=\{(i, \,j):0\leq i\leq 2n+2l-2,0\leq j\leq k\}$.

If $f|_A:$ $A\rightarrow G$ is not an injective map, then similarly, condition (a) holds.

If $f|_A:$ $A\rightarrow G$ is injective, then $|A|=|f(A)|\leq |G|=n(2k+2l+1)+1$.
Since $|A|=(2n+2l-1)(k+1)=n(2k+2l+1)+(2l-1)(k+1-n)$,
we have either $k=n-1$ and $|A|=|G|-1$ or $k=n$, $l=1$ and $|A|=|G|$.
If $k=n-1$ and $|A|=|G|-1$,
the statement $(b)$ holds by a similar argument as in the proof of Lemma \ref{Elemmid}.

If $k=n$, $l=1$ and $|A|=|G|$,
then $|G|=(2n+1)(n+1)$, $A=\{(i, \,j):0\leq i\leq 2n,0\leq j\leq n\}$ and $f(A)=G$.
From this we have
$$-s=is+js'$$ and $$-s'=i_1s+j_1s'$$
where $(i, \,j)$, $(i_1, \,j_1)\in A$.
Now suppose  $(a)$ does not holds, then we must have $i=2n$ and $j_1=n$.
It follows that
\begin{gather}
(2n+1)s+js'=0  \label{eqa}
\end{gather}
and
\begin{gather}
i_1s+(n+1)s'=0.  \label{eqb}
\end{gather}
Since $|G|=(2n+1)(n+1)$,
we get $(n+1)js'=0$ and $(2n+1)i_1s=0$.
The results follows upon multiplying on both sides of Equations (\ref{eqa}) and (\ref{eqb}) by $n+1$ and $2n+1$, respectively.
Once again multiply on both sides of (\ref{eqa}) and (\ref{eqb}) by $i_1$ and $j$, respectively,
and we get $i_1js'=0$ and $i_1js=0$.
Since $s$ and $s'$ generate $G$, then for any $g\in G$, $i_1jg=0$
which says that $|G||i_1j$.
Since $0\leq i_1\leq 2n$ and $0\leq j\leq n$, we have $|G|=(2n+1)(n+1)\leq 2n^2$,
a contradiction. This proves the lemma.
\end{proof}
Applying the above two lemmas, we obtain
\begin{prop} \label{EOthm}
Let $n$, $k_1$ and $k_2$ be integers, $n\geq 2$, $0\leq k_1< k_2$.
If $[-k_1, \,k_2]^*$ splits an abelian group $G$ of order $n(k_1+k_2)+1$,
then $k_1\leq n-2$.
\end{prop}

\begin{proof}
If $k_2-k_1>0$ is even, write $k=k_1$ with $k\geq 0$ and $2l=k_2-k_1$ with $l\geq 1$.
Then the order of $G$ is $|G|=2n(k+l)+1$.

If $n=2$ and  $k\geq 1$, let $S=\{s, \,s'\}$ be the splitter set.
The statement $(a)$ in Lemma \ref{Elemmid} implies that $2n+2l-3>k+2l$, which  contradicts with $k\geq 1$.
If  the statement $(b)$ in Lemma \ref{Elemmid} holds,
then $s'=\pm(2l+2)s$.
Since $(2l+2)\cdot s=\pm s'$,
we must have $k_2=k+2l<2l+2$, and so $k=1$.
For $k_2=k+2l\geq 2$ and $|G|=4(l+1)+1$, we can derive a contradiction from $2\cdot s'=\pm s$.

Now we assume that $n\geq 3$ and suppose $k\geq n-1$.
Let $S=\{s_1, \ldots,  \,s_n\}$ be a splitter set of $G$.
For each index $j$, $2\leq j\leq n$, consider the pair of elements $s_1$ and $s_j$.
Assume that for each such $j$, the statement $(a)$ in Lemma \ref{Elemmid} holds,
that is, there are $x_j$ and $y_j$, $1\leq x_j\leq 2n+2l-3$, $1\leq |y_j|\leq k$, such that
$x_js_1+y_js_j=0$.
If there are $u$, $v$ with $u\neq v$ in $[2, \,n]$ such that $x_u=x_v$,
then $y_us_u=y_vs_v$, a contradiction to the fact that $S$ is a splitter set.
If for any $2\leq u<v\leq n$, $x_u\neq x_v$,
say, $x_2<x_3<\cdots <x_n$,
then $1\leq x_2<x_3<\cdots <x_n\leq 2n+2l-3$, and then $x_2\leq n+2l-1$.
It follows that $-y_2s_2=x_2s_1$ with $-y_2$, $x_2\in [-k,k+2l]^*$, a contradiction.

Thus, there must be an index $j$ such that the statement $(b)$ in Lemma \ref{Elemmid} holds.
It means that $(2n+2l-2)s_1\in S$ or $-(2n+2l-2)s_1\in S$ and $k=n-1$ or $l=1$.
If $(2n+2l-2)s_1\in S$, then $(k+1)\cdot (2n+2l-2)s_1=2n(k+l)s_1=(-1)\cdot s_1$.
If $-(2n+2l-2)s_1\in S$, then $(k+1)\cdot (-(2n+2l-2)s_1)=-2n(k+l)s_1=(1)\cdot s_1$.
Both of above are impossible.

The case $k_2-k_1$ odd follows from Lemma \ref{Olemmid} and the same discussion as in the above proof. This proves the proposition.
\end{proof}

\begin{lem}\label{thmk2}
Let $n$, $k_1$ and $k_2$ be integers with $n\geq 2$, $k_2\geq k_1\geq 1$ and $k_2\geq n-1$.
Suppose that $[-k_1, \,k_2]^*$ splits the cyclic  group $G=\mathbb{Z}_{n(k_1+k_2)+1}$.
Then for any two distinct elements $s$ and $s'$ in a splitter set,  one of the following statements holds:
\begin{description}
  \item[(a)] there are integers $x$ and $y$ satisfying $1\leq x\leq n+k_1-2$, $1\leq y\leq k_2$ and $xs+ys'=0$;
  \item[(b)] $k_2\leq n+k_1-3$ or $k_1=1$, $k_2=n-1$ or $k_1=2$ and $k_2=n$.
\end{description}
\end{lem}

\begin{proof}
Define a map $f:$ $\mathbb{Z}\oplus \mathbb{Z}\rightarrow G$ by $f(i, \,j)=is+js'$.
Put $A=\{(i,j):-k_1\leq i\leq n-2,0\leq j\leq k_2\}$.

If $f|_A: A\rightarrow G$ is not an injective map,
then there are two distinct elements $(i_1,j_1)$, $(i_2,j_2)$ of $A$ such that
$i_1s+j_1s'=i_2s+j_2s'$,  $i_1$,  $i_2\in [-k_1,n-2]$, $j_1,  j_2\in [0,k_2]$. Note that $i_1=i_2$ implies $j_1s'=j_2s', j_1,  j_2\in [0,k_2]$, so $j_1=j_2$, a contradiction. If $j_1=j_2$, then $(i_1-i_2)s=0$, and hence $ord(s)\le |i_1-i_2|\le n+k_1-2\le k_1+k_2-1$, which is also impossible. Therefore we may assume  that $i_1<i_2$ and $j_1\neq j_2$. Thus, $(i_2-i_1)s+(j_2-j_1)s'=0$, where $1\leq i_2-i_1\leq n+k_1-2$ and $1\leq |j_2-j_1|\leq k_2$. If $1\leq j_2-j_1\leq k_2$, then statement $(a)$ holds. If $1\leq j_1-j_2\leq k_2$, then $(i_2-i_1)s=(j_1-j_2)s'$, which implies $k_2\leq n+k_1-3$.

Now suppose that $f|_A: A\rightarrow G$ is injective. Then $|A|=|f(A)|\leq |G|=n(k_1+k_2)+1$.
Since $|A|=(n-1+k_1)(k_2+1)=n(k_1+k_2)+(k_2+1-n)(k_1-1)$, $k_1\ge1$ and $k_2\ge n-1$, we must have that $k_2=n-1$ or $k_1=1$ or $k_1=2$, $k_2=n$.

If $k_1=1$ or $k_2=n-1$, then
 $|A|=|f(A)|=|G|-1$.
Put $$f(A)=G\setminus\{x\}.$$
Then $\{x-s,x+s,x-s',x+s'\}\subseteq G\setminus\{x\}$.
By the same arguments as in the proof Lemma \ref{Elemmid}, we have
$$x-s=f(n-2,j_1)=(n-2)s+j_1s';$$
$$x+s=f(-1,j_2)=-s+j_2s';$$
$$x-s'=f(i_1,k_2)=i_1s+k_2s';$$
$$x+s'=f(i_2,0)=i_2s,$$
where $i_1,$ $i_2\in [-1,n-2]$ and $j_1,$ $j_2\in [0,k_2]$. Hence
 $$x=(n-1)s+j_1s'=-2s+j_2s'=i_1s+(k_2+1)s'=i_2s-s'.$$
It follows from  $(n-1)s+j_1s'=i_1s+(k_2+1)s'$ that
$$(n-1-i_1)s=(k_2+1-j_1)s',$$
where $1\leq n-1-i_1\leq n$ and $1\leq k_2+1-j_1\leq k_2+1$.
If $1\leq k_2+1-j_1\leq k_2$ that is $j_1\in [1,k_2]$,
then $n-1-i_1>k_2$.
Since $k_2\geq n-1$, we obtain that $k_2=n-1$ and $i_1=-1$.
Hence,
$$j_1=0 \ or \ k_2=n-1, \ i_1=-1.$$
Similarly, $-2s+j_2s'=i_2s-s'$ implies $(i_2+2)s=(j_2+1)s'$
with $1\leq i_2+2\leq n$ and $1\leq j_2+1\leq k_2+1$.
Thus
$$j_2=k_2 \ or \ k_2=n-1, \ i_2=n-2.$$
By $-2s+j_2s'=i_1s+(k_2+1)s'$,
we get $(i_1+2)s+(k_2+1-j_2)s'=0$
with $1\leq i_1+2\leq n$ and $1\leq k_2+1-j_2\leq k_2+1$.
If the statement $(a)$ does not hold, then $i_1+2=n$ or $k_2+1-j_2=k_2+1$, that is
$$i_1=n-2 \ or \  j_2=0.$$
From $(n-1)s+j_1s'=i_2s-s'$, we have $(n-1-i_2)s+(j_1+1)s'=0$
with $1\leq n-1-i_2\leq n$ and $1\leq j_1+1\leq k_2+1$.
Hence the statement $(a)$ is not true yields $n-1-i_2=n$ or $j_1+1=k_2+1$, that is
$$i_2=-1 \ or \  j_1=k_2.$$
Hence, $(j_1,j_2,i_1,i_2)=(0,k_2,n-2,-1)$, $-s'=ns$ or $(j_1,j_2,i_1,i_2)=(k_2,0,-1,n-2)$, $k_2=n-1$.
If $-s'=ns$, then $n\geq k_2+1$.
By $k_2\geq n-1$, $k_2=n-1$.
This completes the proof.
\end{proof}

\begin{lem}{\rm (\cite{[PK]}Theorems 4.3-4.4) }\label{lemH3}
Let $p$ be an odd prime with $p\equiv 1$ (mod $4$). Then there exists a perfect $B[-1, \,3](p)$ set
if and only if $6$ is a quartic residue \textit{modulo} $p$, if $p\equiv 5$ (mod $8$);
$o(-\frac{3}{2})$ is odd and $4||<-1, \, 2, \,3>|$, if $p\equiv 1$ (mod $8$).
\end{lem}

\begin{lem}{\rm (\cite{[H]}, Theorem 2.2.3)}\label{nonsingular}
Let $G$ be a finite group and $M$ a set of nonzero integers.
Then $M$ splits $G$ nonsingularly if and only if $M$ splits $\mathbb{Z}_p$ for each prime divisor $p$ of $|G|$.
\end{lem}

For a finite group $G$, if the size of its splitter set $S$ is $2$, then it is easy to prove the following result.
\begin{lem} \label{cyclic-2}
Let $k_1$ and $k_2$ be integers with  $0\le k_1< k_2$ and let $M=[-k_1, \,k_2]^*$.
If $M$ splits an abelian  group $G$ of order $2|M|+1$,
then  $G$ is cyclic and $k_1=0$.
\end{lem}

\begin{proof}
Let $S=\{g_1, \,g_2\}$ be  a splitter set,   then $G\setminus \{0\}=[-k_1, \,k_2]^\ast\cdot \{g_1, \,g_2\}$.
Suppose  $G$ is not a cyclic group, we see that $g_1+g_2\not\in [-k_1, \,k_2]^\ast\cdot \{g_1, \,g_2\}$, a contradiction.

Since $G$ is a cyclic group, without loss of generality, we may assume that $G=\mathbb{Z}_{2|M|+1}$ and  $S=\{1, \,a\}$ is a splitter set. If $k_1>0$, then the statement $(a)$ in Lemma \ref{thmk2} says that there are integers $x$ and $y$  such that $1\le x\le n+k_1-2=k_1$, $1\le y\le k_2$ and $x+ya=0$, that is $y=-x\cdot a$ with $-x, \,y\in [-k_1,k_2]^\ast$,  which is impossible.

By the statement $(b)$ in Lemma \ref{thmk2}, we have $k_2\leq k_1-1$ or $k_1=k_2=1$ or $k_1=k_2=2$. This contradicts with $k_1<k_2$. Therefore $k_1=0$ and $M=[1, k_2]$.

\end{proof}

\begin{prop} \label{thmin1}
Let $n$, $k_1$ and $k_2$ be integers with $n\geq 2$ and $1\leq k_1< k_2$.
If $[-k_1,k_2]^*$ splits a cyclic group $G$ of order $n(k_1+k_2)+1$,
then $k_2\leq 2n-5$.
\end{prop}

\begin{proof}
If $n=2$, the result follows immediately from Lemma \ref{cyclic-2}.

For $n\geq 3$ and  $k_2\geq n-1$, following  the argument in the proof of Proposition \ref{EOthm}, we see that the statement $(a)$ in Lemma \ref{thmk2} does not hold.

Suppose that the statement $(b)$ in Lemma \ref{thmk2} holds.
Then $k_2\leq n+k_1-3$ or $k_1=1$, $k_2=n-1$ or $k_1=2$, $k_2=n$.
By Proposition \ref{EOthm}, $k_2\leq n+k_1-3$ implies $k_2\leq 2n-5$.
For $(k_1, \,k_2)=(1, \,n-1)$ or $(2, \,n)$,
if $n\geq 5$, it is easily seen that $k_2\leq 2n-5$.
If $n=4$, then we infer that $[-1, \,3]^*$ splits the cyclic group $\mathbb{Z}_{17}$
or $[-2, \,4]^\ast$ splits an abelian group $G$ of order $25$.
However, Lemma \ref{lemH3} yields that $[-1, \,3]^*$ can not split $\mathbb{Z}_{17}$.
In addition, Lemma \ref{nonsingular} yields that $[-2, \,4]^*$ splits the cyclic group $\mathbb{Z}_5$, which is impossible.

If $n=3$, by Proposition \ref{EOthm}, $k_1\leq n-2$ implies $(k_1, \,k_2)=(1, \,2)$.
It follows from the proof of Lemma \ref{thmk2} that $|G|=n(k_1+k_2)+1=10$, so $[-1, 2]^\ast$ splits the cyclic group $\mathbb{Z}_5$ by Lemma \ref{nonsingular}, which is impossible.
This completes the proof.
\end{proof}

\textit{The proofs of Theorem \ref{thm1}:} The proof follows immediately from Propositions \ref{EOthm} and \ref{thmin1}.

\section{Proofs of Theorems 1.3-1.5}

In this section, we will prove Theorems 1.3, 1.4 and 1.5. We first prove Theorem 1.3.

{\bf Proof of Theorem \ref{primeB}:}
Since $\gcd(s,m)=1$ for all $s\in S$ and the splitting is purely singular,  then for any prime $p|m$ we have $p\le k_2$,  and for any $d|m, \, 1<d<m$,
$$\{g\in \mathbb{Z}_{m}\setminus \{0\}: d|g\}=M_{d}\cdot S, $$
where $M_d=\{i\in [-k_1, \,k_2]^\ast: d|i\}$.
Hence $d\le k_2$ and $|\{g\in \mathbb{Z}_{m}\setminus \{0\}: d|g\}|=\frac{m}{d}-1=|M_{d}|\cdot |S|=([\frac{k_1}{d}]+[\frac{k_2}{d}])\cdot |S|$.
Let $k_1=u_1d+v_1, v_1\in [0, d-1]$ and $k_2=u_2d+v_2, v_2\in [0, \,d-1]$, then $u_1=[\frac{k_1}{d}]$ and $u_2=[\frac{k_2}{d}]$.
Form $m=(k_1+k_2)|S|+1$, we have
$$(u_1+u_2)|S|=\frac{((u_1+u_2)d+(v_1+v_2))|S|+1}{d}-1=(u_1+u_2)|S|+\frac{(v_1+v_2)|S|+1}{d}-1.$$
Since $d\le k_2$, we see that $$(v_1+v_2)|S|=d-1\leq k_2-1,$$ which implies that $v_1+v_2\geq 1.$

If $k_1=k_2$, then $|S|\geq 2$ and Lemma \ref{cross} implies $k_2\leq |S|-1$.
Hence $|S|\leq (v_1+v_2)|S|=d-1\leq k_2-1\leq |S|-2$, a contradiction.
Therefore $|S|=1$ and $m=k_1+k_2+1$.  If $k_1=0$ and $|S|\geq 3$, then $k_2\leq |S|-2$ by Lemma \ref{semi-cross}.
Hence $|S|\leq (v_1+v_2)|S|\le k_2-1\leq |S|-3$, again a contradiction. If $k_1=0$ and $|S|=2$, it follows that $m=2k_2+1$.

If $1\leq k_1<k_2$ and $|S|\geq 2$, by Proposition \ref{thmin1} we have $k_2\leq 2|S|-5$.
Hence $$|S|\leq (v_1+v_2)|S|=d-1\leq k_2-1\leq 2|S|-6.$$
It follows that $v_1+v_2=1$ and $|S|=d-1$ for any $d|m, 1<d<m$. This means that $m$ has only one positive divisor other than $1$ and $m$, so $m=p^2$ for some prime $p$. In this case, we have $k_1=1$, $k_2=p$ and $|S|=p-1$,
$$pS=\{p\in\mathbb{Z}_{p^2}\setminus\{0\}, p|g\}=\{ip\pmod{p^2}, i\in[1, p-1]\}.$$
Hence we may assume that $$S=\{l_ip+i|1\leq i\leq p-1 \ and \ 1\leq l_i\leq p\}.$$
Let $$l_{p-1}=p-j$$ for some $j\in[1,p]$. If $j=1$, then $l_{p-1}p+p-1\equiv-1\pmod{p^2}\in S$, a contradiction. If $j\in[2, p-1]$, let $y$ be the least positive integer modulo $p$ such that $(p-1)y\equiv1\pmod{p}$,  then $(l_{p-1}p+p-1)(p-y)\equiv (j-1)yp-p+y\equiv y \pmod{p^2}$. This means $y$ has two different representations in $[-1, \,p]^\ast\cdot S$, again a contradiction. This completes the proof.
\qed

To prove Theorem 1.4, we also need the following result for the splittings of cyclic groups.

\begin{lem}[\cite{[SS]}, Theorem 3.2] \label{p-group-sp}
If $p$ is an odd prime and $\mathbb{Z}_{p^{\alpha}}\setminus \{0\}=MS$ is a splitting, then
either $M$ or $S$ contains only elements relatively prime to $p$.
\end{lem}

\begin{coro} \label{cyclic-p-gr}
Let $\alpha$, $k_1$, $k_2$ be integers, $\alpha \geq 2$, $0\leq k_1\leq k_2$ and $k_2\ge3$,
and let $p_0$, $p$, $q$ be primes, $ p_0\ne2$.
Suppose that $[-k_1, \,k_2]^*$ splits a cyclic group $\mathbb{Z}_{m}$.
If the splitting is purely singular and $m=p_0^{\alpha}$ or $pq$,
then either $m=k_1+k_2+1$ or $k_1=0$  and $m=2k_2+1$.
\end{coro}

Let $S$ be the splitter set. By Theorem \ref{primeB}, it suffices to show that $\gcd(s, \, m)=1$ for all $s\in S$.
If $m=p_0^{\alpha}$, since the splitting is purely singular, so
it follows from Lemma \ref{p-group-sp} that $gcd(s,m)=1$ for all $s\in S$.

Put $$M_{\ell}=\{k\in [-k_1,k_2]^*: \ell|k\}, \quad S_{\ell}=\{s\in S: \ell|s\}.$$
If $m=pq$, since the splitting is purely singular,  so we have $|M_p|>0$ and $|M_q|>0$.
Hence $|S_p|=|S_q|=0$,
for otherwise $0$ is contained in $[-k_1,k_2]^*\cdot S=\mathbb{Z}_{m}\setminus \{0\}$. It follows that $\gcd(s, \, m)=1$ for all $s\in S$.
This completes the proof.
\qed

{\bf Proof of Theorem 1.4:}

For $\alpha=0$, the result follows immediately from  Corollary \ref{cyclic-p-gr}.

If $\alpha>0$, since $m\equiv1\pmod{k_1+k_2}$, so $k_1+k_2$ is odd and $2|k_1+k_2+1$.
It follows that $2|(k_1+k_2+1,\,m)>1$.
By Lemma \ref{gen-pk-lem}, we obtain that $k_1 + k_2 + 1|m$ and $\gcd(k_1 + k_2 + 1, \frac{m}{k_1+k_2+1}) = 1$.
Thus  $2^{\alpha}|k_1 + k_2 + 1$ and $\frac{m}{k_1+k_2+1}=p^\gamma$ or $p_1p_2$. Recall that $[-k_1, \,k_2]^\ast$ splits both $\mathbb{Z}_{m}$ and $\mathbb{Z}_{k_1+k_2+1}$. It follows from Lemma 2.3 that $[-k_1, \,k_2]^\ast$ splits $\mathbb{Z}_{\frac{m}{k_1+k_2+1}}$, so $\frac{m}{k_1+k_2+1}=k_1+k_2+1$ by Corollary \ref{cyclic-p-gr}, which contradicts with Lemma \ref{gen-pk-lem}.
This completes the proof of Theorem 1.4.
\qed

To prove Theorem 1.5, we need some other results. The following result follows immediately from a similar argument as in Lemma 15 \cite{[ZG]}.  We have
\begin{lem}\label{BjiaoM}
Let $n$, $k_1$ and $k_2$ be positive integers with $0\leq k_1< k_2$.
Suppose $B$ is a perfect $B[-k_1, \,k_2](n)$ set.
Set $\mathbb{Z}'_n=\{i:i\in \mathbb{Z}_n, \,\gcd(i, \,n)=1\}$,
$M(n)=\{i: i\in [-k_1, \,k_2]^*, \,\gcd(i,n)=1\}$
and $B(n)=\{i: i\in B, \,\gcd(i, \,n)=1\}$.
Then $\mathbb{Z}'_n=M(n)\cdot B(n)$.
\end{lem}

We also need the following result.
\begin{prop}\label{po1}Let $G$ be an abelian group (written multiplicatively)  with $|G|=2m, m\in\mathbb{N}$. Suppose $N$ is a subset of $G$ such that
$\{1, a\}\subseteq N$, $a\ne1, a^2=1$, where $1$ denote the unity of the group $G$,  and $|N|$ is odd, then $N$ is not a direct factor of $G$. \end{prop}
\begin{proof} If $N$ is a direct factor of $G$, then there exists a sunset $A$ of $G$ such that $N\cdot A=G$ is a factorization. Since $|N|$ is odd, by Proposition 2.1, $N^2\cdot A $ is also a factorization of $G$, which implies that $|N^2|=|N|$. However, $1^2=a^2=1$ in $G$, it follows that $|N^2|\le|N|-1$, a contradiction. \end{proof}

{\bf Proof of Theorem 1.5:}
Recall that  Hickerson have  verified  the theorem  for all $[1, \, k]$ with  $k < 3000$. Moreover, Zhang and Ge \cite{[ZG]} solved the case $[-1, \,k]^\ast$ when $k=3, \,4, \,5$, $6, \,8$, $9, \,10$ and $[-2, \,k]^\ast$ when $k=3, \,4, \,6$. So we need only consider the case with $k_1>0$.
By Theorem 1.4 and Lemma 2.5, to prove Theorem 1.5, it suffices to  show that $[-k_1, \, k_2]^*$ does not split the cyclic group $\mathbb{Z}_m$ with $\gcd(m, \,(k_1+k_2)(k_1+k_2+1))=1$,  $m$ has at least two distinct odd prime divisors and $m$ has no prime divisor greater than $k_2$. Now we prove the theorem case by case.

For $k_1+k_2\le6$, there is not any $m$ with the above property.

$\bullet$ $k_1+k_2=7$, $[-k_1, \,k_2]^*=[-2, \,5]^*$. For this case, $m=3^\alpha5^\beta, \, \alpha, \, \beta\in\mathbb{N}$. By Lemma \ref{BjiaoM}, $\mathbb{Z}_n'=B\cdot M$, where $M=\{-2, \, -1, \, 1, \, 2,\, 4\}$, is a factorization. Since $|M|=5$ is odd, $(-1)^2=1$ and $|\mathbb{Z}_n'|=\varphi(m)$ is even, so it is impossible by Proposition \ref{po1}. Thus there does not exist any purely singular perfect $B[-2, \,5](m)$ set except for $m=1, 8$.

$\bullet$ $k_1+k_2=8$, $[-k_1, \,k_2]^*=[-1, \,7]^*$. For this case, $m=5^\alpha7^\beta, \, \alpha, \, \beta\in\mathbb{N}$. Put
$$B_0=\{x\in B, \gcd(x, \,m)=1\}, \quad B_5=\{x\in B, \gcd(x, \,m)=5\}, \quad B_7=\{x\in B, \gcd(x, \,m)=7\}.$$
By calculating the number of the elements $x$ of $\mathbb{Z}_m$ with $\gcd(x, \,m)=1$, $\gcd(x, \,m)=5$ and $\gcd(x, \,m)=7$, respectively, we get
$$ 6|B_0|=\varphi(m)=4\cdot6\cdot5^{\alpha-1}7^{\beta-1},$$
so $|B_0|=4\cdot5^{\alpha-1}7^{\beta-1}$. If $\alpha>1$, then we have
$$6|B_5|+|B_0|=\varphi(\frac{m}{5})=4\cdot6\cdot5^{\alpha-2}7^{\beta-1},$$
which implies that $3||B_0|$, a contradiction. If $\beta>1$, then we have
$$6|B_7|+|B_0|=\varphi(\frac{m}{7})=4\cdot6\cdot5^{\alpha-1}7^{\beta-2},$$
which also implies that $3||B_0|$, again a contradiction.  Hence $m=35\not\equiv1\pmod{8}$, which is impossible. Thus there does not exist any purely singular perfect $B[-1, \,7](m)$ set except for $m=1, 9$.

$\bullet$ $k_1+k_2=9$. Then $\gcd(m, \, 30)=1$, and there are no such $m$ satisfies the required properties.
Hence, there does not exist any purely singular perfect $B[-k_1, \,k_2](m)$ set except for $m=1, 10$.

$\bullet$ $k_1+k_2=10$, $[-k_1, \,k_2]^*=[-2, \,8]^*$. For this case, $m=3^\alpha7^\beta, \, \alpha, \, \beta\in\mathbb{N}$. By Lemma \ref{BjiaoM}, $\mathbb{Z}_n'=B\cdot M$, where $M=\{-2, \, -1, \, 1, \, 2,\, 4, \, 5, \,8\}$, is a factorization. Since $|M|=7$ is odd, $(-1)^2=1$ and $|\mathbb{Z}_n'|=\varphi(m)$ is even, so it is impossible by Proposition \ref{po1}. Thus there does not exist any purely singular perfect $B[-2, \,8](m)$ set except for $m=1$.

$\bullet$ $k_1+k_2=10$, $[-k_1, \,k_2]^*=[-3, \,7]^*$. For this case, $m=3^\alpha7^\beta, \, \alpha, \, \beta\in\mathbb{N}$. Put
$$B_0=\{x\in B, \gcd(x, \,m)=1\}, \quad B_3=\{x\in B, \gcd(x, \,m)=3\}, \quad B_7=\{x\in B, \gcd(x, \,m)=7\}.$$
By calculating the number of the elements $x$ of $\mathbb{Z}_m$ with $\gcd(x, m)=1$, $\gcd(x, \,m)=3$ and $\gcd(x, \,m)=7$, respectively, we get
$$ 6|B_0|=\varphi(m)=2\cdot6\cdot3^{\alpha-1}7^{\beta-1},$$
so $|B_0|=2\cdot3^{\alpha-1}7^{\beta-1}$. If $\alpha>1$, then we have
$$6|B_5|+3|B_0|=\varphi(\frac{m}{3})=2\cdot6\cdot3^{\alpha-2}7^{\beta-1}<3|B_0|,$$
a contradiction. If $\beta>1$, then we have
$$6|B_7|+|B_0|=\varphi(\frac{m}{7})=2\cdot6\cdot3^{\alpha-1}7^{\beta-2}<|B_0|,$$
 again a contradiction.  Hence $m=21\not\equiv1\pmod{8}$, which is impossible. Thus there does not exist any purely singular perfect $B[-3, \,7](m)$ set except for $m=1$.

$\bullet$ $k_1+k_2=11$, $[-k_1, \,k_2]^*=[-2, \,9]^*$. For this case, $m=5^\alpha7^\beta, \, \alpha, \, \beta\in\mathbb{N}$. By Lemma \ref{BjiaoM}, $\mathbb{Z}_n'=B\cdot M$, where $M=\{-2, \, -1, \, 1, \, 2, \,3, \, 4, \, 6, \,8, \,9\}$, is a factorization. Since $|M|=9$ is odd, $(-1)^2=1$ and $|\mathbb{Z}_n'|=\varphi(m)$ is even, so it is impossible by Proposition \ref{po1}. Thus there does not exist any purely singular perfect $B[-2, \,9](m)$ set except for $m=1, 12$.

$\bullet$ $k_1+k_2=11$, $[-k_1, \,k_2]^*=[-3, \,8]^*$ or $[-4, \,7]^*$. The argument is the same as the proof of the case $[-k_1, \,k_2]^*=[-2, \,9]^*$.

$\bullet$ $k_1+k_2=12$, $[-k_1, \,k_2]^*=[-1, \,11]^*$. For this case, $m=5^\alpha7^\beta11^\gamma, \, \alpha, \, \beta, \gamma\in\mathbb{N}$ or $m=5^\alpha11^\gamma, \, \alpha, \, \gamma\in\mathbb{N}$ or $m=7^\beta11^\gamma, \,  \beta, \, \gamma\in\mathbb{N}$ or $m=5^\alpha7^\beta, \, \alpha, \, \beta\in\mathbb{N}$. If $m=5^\alpha7^\beta11^\gamma, \, \alpha, \, \beta, \gamma\in\mathbb{N}$, let $S_{11}=\{x\in B, \, \gcd(x, \,11)=1\}$, by calculating the number of the elements $x$ of $\mathbb{Z}_m$ with $\gcd(x, 11)=1$, we obtain $11S_{11}= 10\cdot 5^\alpha7^\beta11^{\gamma-1}$, which implies that $\gamma>1$. Put
$$B_0=\{x\in B, gcd(x, \,m)=1\}, \quad B_{11}=\{x\in B, gcd(x, \,m)=11\}.$$
By calculating the number of the elements $x$ of $\mathbb{Z}_m$ with $\gcd(x, m)=1$ and $\gcd(x, \,m)=11$, respectively, we get
 $$8|B_0|=240\cdot5^{\alpha-1}7^{\beta-1}11^{\gamma-1},$$
 so $|B_0|=30\cdot5^{\alpha-1}7^{\beta-1}11^{\gamma-1}$. Since $\gamma>1$, then we have
 $$|B_0|+8|B_{11}|=\varphi(\frac{m}{11})=240\cdot5^{\alpha-1}7^{\beta-1}11^{\gamma-2},$$ therefore $8||B_0|=30\cdot5^{\alpha-1}7^{\beta-1}11^{\gamma-1}$, a contradiction. If $m=5^\alpha11^\gamma, \, \alpha, \, \gamma\in\mathbb{N}$ or $m=5^\alpha7^\beta, \, \alpha, \, \beta\in\mathbb{N}$, let
 $$B_0=\{x\in B, gcd(x, m)=1\}.$$
 Calculating the number of the elements $x$ of $\mathbb{Z}_m$ with $\gcd(x, m)=1$, we get $9|B_0|=40\cdot5^{\alpha-1}11^{\gamma-1}$ or $9|B_0|=24\cdot5^{\alpha-1}7^{\beta-1}$, which is impossible. If $m=7^\beta11^\gamma, \,  \beta, \, \gamma\in\mathbb{N}$, let $S_{11}=\{x\in B, \, \gcd(x, \,11)=1\}$, by calculating the number of the elements $x$ of $\mathbb{Z}_m$ with $\gcd(x, 11)=1$, we obtain $11S_{11}= 10\cdot 7^\beta11^{\gamma-1}$, which implies that $\gamma>1$. Put
$$B_0=\{x\in B, gcd(x, \,m)=1\}, \quad B_{11}=\{x\in B, gcd(x, \,m)=11\}.$$
By calculating the number of the elements $x$ of $\mathbb{Z}_m$ with $\gcd(x, m)=1$ and $\gcd(x, \,m)=11$, respectively, we get
 $$10|B_0|=60\cdot7^{\beta-1}11^{\gamma-1},$$
 so $|B_0|=6\cdot7^{\beta-1}11^{\gamma-1}$. Since $\gamma>1$, then we have
 $$|B_0|+10|B_{11}|=\varphi(\frac{m}{11})=60\cdot5^{\alpha-1}7^{\beta-1}11^{\gamma-2}<|B_0|,$$  a contradiction. Thus there does not exist any purely singular perfect $B[-1, \,11](m)$ set except for $m=1$.

$\bullet$ $k_1+k_2=12$, $[-k_1, \,k_2]^*=[-2, \,10]^*$, $[-3, 9]^\ast$, $[-4, \,8]^*$ or $[-5, \,7]^\ast$. For these cases, $m=5^\alpha7^\beta, \, \alpha, \, \beta\in\mathbb{N}$. Put
$S_7=\{x\in B, \, \gcd(x, \,7)=1\}$, by calculating the number of the elements $x$ of $\mathbb{Z}_m$ with $\gcd(x, \,11)=1$, we obtain $11S_7= 6\cdot 5^\alpha7^{\beta-1}$, which is impossible.
Thus there does not exist any purely singular perfect $B[-k_1, \,k_2](m)$ set except for $m=1$.

$\bullet$ $k_1+k_2=13$, $[-k_1, \,k_2]^*=[-1, \,12]^*$. For this case, $m=3^\alpha5^\beta11^\gamma, \, \alpha, \, \beta, \gamma\in\mathbb{N}$ or $m=3^\alpha11^\gamma, \, \alpha, \, \gamma\in\mathbb{N}$ or $m=5^\beta11^\gamma, \,  \beta, \, \gamma\in\mathbb{N}$ or $m=3^\alpha5^\beta, \, \alpha, \, \beta\in\mathbb{N}$. If $m=3^\alpha5^\beta11^\gamma, \, \alpha, \, \beta, \gamma\in\mathbb{N}$ or $m=3^\alpha11^\gamma, \, \alpha, \, \gamma\in\mathbb{N}$ or $m=5^\beta11^\gamma, \,  \beta, \, \gamma\in\mathbb{N}$, let $S_{11}=\{x\in B, \, \gcd(x, 11)=1\}$, by calculating the number of the elements $x$ of $\mathbb{Z}_m$ with $\gcd(x, 11)=1$, we obtain $12S_{11}= 10\cdot 3^\alpha5^\beta11^{\gamma-1}$ or $12S_{11}= 10\cdot 3^\alpha11^{\gamma-1}$ or $12S_{11}= 10\cdot5^\beta11^{\gamma-1}$, which is impossible. If $m=3^\alpha5^\beta, \, \alpha, \, \beta\in\mathbb{N}$, let $B_0=\{x\in B, gcd(x, m)=1\}$, then we calculate the number of the elements $x$ of $\mathbb{Z}_m$ with $\gcd(x, m)=1$, we obtain $7S_0= 8\cdot 3^{\alpha-1}5^{\beta-1}$, which is impossible. Thus there does not exist any purely singular perfect $B[-1, \,12](m)$ sets except for $m=1, 14$.

$\bullet$ $k_1+k_2=13$, $[-k_1, \,k_2]^*=[-2, \,11]^*$. For this case, $m=3^\alpha5^\beta11^\gamma, \, \alpha, \, \beta, \gamma\in\mathbb{N}$ or $m=3^\alpha11^\gamma, \, \alpha, \, \gamma\in\mathbb{N}$ or $m=5^\beta11^\gamma, \,  \beta, \, \gamma\in\mathbb{N}$ or $m=3^\alpha5^\beta, \, \alpha, \, \beta\in\mathbb{N}$. If $m=3^\alpha5^\beta11^\gamma, \, \alpha, \, \beta, \gamma\in\mathbb{N}$, let $B_0=\{x\in B, gcd(x,   \,m)=1\}$, then we calculate the number of the elements $x$ of $\mathbb{Z}_m$ with $\gcd(x, m)=1$, we obtain $7|B_0|= 80\cdot 3^{\alpha-1}5^{\beta-1}11^{\gamma-1}$, which is impossible. If $m=3^\alpha11^\gamma, \, \alpha, \, \gamma\in\mathbb{N}$ or $m=5^\beta11^\gamma, \,  \beta, \, \gamma\in\mathbb{N}$, let $S_{11}=\{x\in B, \, \gcd(x, 11)=1\}$, by calculating the number of the elements $x$ of $\mathbb{Z}_m$ with $\gcd(x, 11)=1$, we obtain $12S_{11}= 10\cdot 3^\alpha11^{\gamma-1}$ or  $12S_{11}= 10\cdot5^\beta11^{\gamma-1}$, which is impossible. If $m=3^\alpha5^\beta, \, \alpha, \, \beta\in\mathbb{N}$, let $S_5=\{x\in B, \gcd(x, 5)=1\}$, then we calculate the number of the elements $x$ of $\mathbb{Z}_m$ with $\gcd(x, 5)=1$, we obtain $11|S_5|= 4\cdot 3^\alpha5^{\beta-1}$, which is impossible.
Thus there does not exist any purely singular perfect $B[-2, \,11](m)$ sets except for $m=1, 14$.

$\bullet$ $k_1+k_2=13$, $[-k_1, \,k_2]^*=[-3, \,10]^*$. For this case, $m=3^\alpha5^\beta, \, \alpha, \, \beta\in\mathbb{N}$, let $B_0=\{x\in B, gcd(x, m)=1\}$, then we calculate the number of the elements $x$ of $\mathbb{Z}_m$ with $\gcd(x, m)=1$, we obtain $7S_0= 8\cdot 3^{\alpha-1}5^{\beta-1}$, which is impossible.
Thus there does not exist any purely singular perfect $B[-3, \,10](m)$ sets except for $m=1, 14$.

$\bullet$ $k_1+k_2=13$, $[-k_1, \,k_2]^*=[-4, \,9]^*$ or $[-5, 8]^\ast$. For these two cases, $m=3^\alpha5^\beta$,  $\alpha$,  $\beta\in\mathbb{N}$.  Put
$$B_0=\{x\in B, gcd(x, m)=1\}, \quad B_3=\{x\in B, gcd(x, m)=3\}, \quad B_5=\{x\in B, gcd(x, m)=5\}.$$
By calculating the number of the elements $x$ of $\mathbb{Z}_m$ with $\gcd(x, m)=1$, $\gcd(x, m)=3$ and $\gcd(x, m)=5$, respectively, we get
$$ 8|B_0|=\varphi(m)=8\cdot3^{\alpha-1}5^{\beta-1},$$
so $|B_0|=3^{\alpha-1}5^{\beta-1}$. If $\alpha>1$, then we have
$$8|B_3|+3|B_0|=\varphi(\frac{m}{3})=8\cdot3^{\alpha-2}5^{\beta-1}<3|B_0|,$$
a contradiction. If $\beta>1$, then we have
$$8|B_5|+\sigma|B_0|=\varphi(\frac{m}{5})=8\cdot3^{\alpha-1}5^{\beta-2},$$
where $\sigma=1$ if $[-k_1, \,k_2]^*=[-4, \,9]^*;$ $\sigma=2$ if $[-k_1, \,k_2]^*=[-5, 8]^\ast$.
This implies that $2||B_0|$, a contradiction.
Hence $m=15\not\equiv1\pmod{13}$, which is impossible. Thus there does not exist any purely singular perfect $B[-4, \,9](m)$ or $B[-5, 8](m)$ sets except for $m=1, 14$.

 $\bullet$ $k_1+k_2=13$, $[-k_1, \,k_2]^*=[-6, \,7]^*$. For this case, $m=3^\alpha5^\beta, \, \alpha, \, \beta\in\mathbb{N}$, let $B_0=\{x\in B, gcd(x, m)=1\}$, then we calculate the number of the elements $x$ of $\mathbb{Z}_m$ with $\gcd(x, m)=1$, we obtain $7S_0= 8\cdot 3^{\alpha-1}5^{\beta-1}$, which is impossible.
Thus there does not exist any purely singular perfect $B[-6, \,7](m)$ sets except for $m=1, 14$.

 $\bullet$ $k_1+k_2=14$, $[-k_1, \,k_2]^*=[-1, \,13]^*$. For this case, $m=11^\alpha13^\beta, \, \alpha, \, \beta\in\mathbb{N}$,  put
$B_0=\{x\in B, gcd(x, m)=1\}, \quad B_{11}=\{x\in B, \gcd(x, \,m)=11\}, \quad B_{13}=\{x\in B, \gcd(x, m)=13\}$.
By calculating the number of the elements $x$ of $\mathbb{Z}_m$ with $\gcd(x, m)=1$, $\gcd(x, m)=11$ and $\gcd(x, m)=13$, respectively, we get
$$ 12|B_0|=\varphi(m)=120\cdot11^{\alpha-1}13^{\beta-1},$$
so $|B_0|=10\cdot11^{\alpha-1}13^{\beta-1}$. If $\alpha>1$, then we have
$$12|B_{11}|+|B_0|=\varphi(\frac{m}{11})=120\cdot11^{\alpha-2}13^{\beta-1},$$
which implies that $12||B_0|$, a contradiction. If $\beta>1$, then we have
$$12|B_{13}|+|B_0|=\varphi(\frac{m}{13})=120\cdot11^{\alpha-1}13^{\beta-2}<|B_0|,$$
 again a contradiction.  Hence $m=143\not\equiv1\pmod{14}$, which is impossible. Thus there does not exist any purely singular perfect $B[-1, \,13](m)$  sets except for $m=1, 15$.
This completes the proof of Theorem 1.5.

\qed

{\bf Remark:} It is easy to see that we can prove more by the method in the proof of Theorem 1.5, since there are many cases have to be discussed, we stop here.

Finally, by Theorem \ref{primeB}, we propose the following conjecture which implies both Conjecture \ref{conj} and  Conjecture \ref{conj1}.

\begin{conj} \label{conj2}
Let $k_1, k_2$ be integers with $0 \le k_1 \leq k_2$ and
$k_1 + k_2\ge 4$.
If there exists a purely singular perfect $B[-k_1,k_2](m)$ set with the splitter set $S$,
then $\gcd(s, \,m)=1$ for all $s\in S$.
\end{conj}




\end{document}